\documentclass[english]{smfart}
\usepackage[francais,english]{babel} 
\usepackage{smfthm}
\usepackage{proof} 
\usepackage{amsmath}
\usepackage{amsfonts}
\newcommand\imp\rightarrow 
\newcommand\nn{^{\lnot\lnot}}
\newcommand\LJ{\mbox{\textsc{nj}}}
\newcommand\LK{\mbox{\textsc{nk}}}
\newcommand{\seq}{\vdash}

\usepackage[utf8]{inputenc} %
\usepackage[T1]{fontenc} %

\theoremstyle{plain}

\newtheorem{theorem}{Theorem}
\newtheorem{remark}[theorem]{Remark}
\newtheorem{lemma}[theorem]{Lemma}
\newtheorem{proposition}[theorem]{Proposition}

\author{Richard Moot}
\address{CNRS, LaBRI, Bordeaux}
\email{richard.moot@labri.fr}
\urladdr{http://www.labri.fr/perso/moot} 

\author{Christian Retoré}
\address{Université de Montpellier \& Texte/LIRMM-CNRS}
\email{christian.retore@umontpellier.fr}
\urladdr{http://www.lirmm.fr/~retore}

\title[Classical logic and intuitionistic logic]{Classical logic and intuitionistic logic: equivalent formulations in natural deduction,  
Gödel-Kolmogorov-Glivenko translation}

\alttitle{Logique classique et intuitionniste: formulatiosn équivalentes en déduction naturelle, traduction de G\"odel-Kolmogorov-Glivenko}

\begin{document} 

\frontmatter
\begin{abstract}
This report first shows the equivalence between several
  formulations of classical logic in intuitionistic logic (tertium non datur, reductio ad absurdum, Pierce's law). 
Then it establishes the correctness of the G\"odel-Kolmogorov translation, whose restriction to the propositional case is due to  Glivenko.  This translation maps a formula $F$ of first order logic to a formula $F^{\lnot\lnot}$ in such a way that $F$ is provable in classical logic if and only if $F^{\lnot\lnot}$ is provable in intuitionistic logic. All formal proofs are presented in natural deduction.

These questions are well-known proof theoretical  facts, but in textbooks, they are often ignored or left to the reader. 
Because of the combinatorial difficulty of some of the  needed formal proofs, we hope that this report may be useful, in particular to students and colleagues from other areas. \end{abstract}
\begin{altabstract}
Ce rapport commence par établir l'équivalence entre diverses formulation de la logique classique (tiers exclu, raisonnement par l'absurde, loi de Pierce) en logique intuitionniste. Nous  montrons ensuite la correction de la traduction de G\"odel-Kolmogorov, dont la restriction au cas propositionnel est due à Glivenko.  Cette traduction associe à toute formule $F$ de la logique du premier ordre une formule $F^{\lnot\lnot}$ 
telle que $F$ est démontrable en logique classique si et seulement si $F^{\lnot\lnot}$  est démontrable en logique intuitionniste. Toutes les preuves formelles sont présentées en déduction naturelle. 

Ces résultats de théorie de la démonstration sont bien connus, mais dans les ouvrages où ils sont mentionnés, leurs démonstrations sont souvent laissées en exercice au lecteur. Vue la difficulté combinatoire de certains des cas  de ces démonstrations, nous pensons que ce rapport pourra être utile aux étudiants et aux collègues d'autres domaines. 
\end{altabstract}
\altkeywords{logique, théorie de la démonstration, logique classique, logique intuitionniste, déduction naturelle}

\maketitle

\mainmatter

\section{Foreword and references}

The results in here are not new, 
but it is hard to tell where they are properly written down and published.
They are often left as exercices to the reader. These are exercices 
on the combinatorics of proofs, but some of the cases are not that easy 
for people who are new to proof theory, especially students (though we
encourage everyone to try these proofs themselves). 

Natural deduction is a tree-like framework for formal proofs which is naturally intuitionistic. 
This  tree-like formulation  was introduced in
\cite{Pra65}, where NJ \cite{Gen34,Gen34b} is reformulated  in terms of pseudo-trees. 
More modern references include \cite{GLT88,vanDalen2013}. 

The translation of a formula $F$  into a formula $F^{\lnot\lnot}$
which is intituitonistically provable if and only if $F$ is classically provable is due to 
\cite{Glivenko29} (in French) 
for the propositional case and to 
\cite{Kolmogorov1925} (in Russian, summary in French, in English in \cite{Frege2Goedel1967}) 
and 
\cite{Goedel1933HA,Goedel33nonnon} (in German) 
for the first order case. 

The equivalence of the various formulations of classical logic in intuitionistic logic (tertium no datur, 
reductio ad absurdum, Pierce law) can be found here and there e.g.\ in 
\cite{RasiowaSikorski1963,TD88,Mints2000}


\section{Natural deduction rules}

Let us recall the natural deduction rules that we use throughout this report. 

A proof is a tree plus additional information: 
\begin{itemize} 
\item the nodes are formul\ae\ 
\item the root is the conclusion of the proof
\item the leaves are the hypothesis which can be:
\begin{itemize} 
\item cancelled or discharged (if so, they are between square brackets)
\item free (nothing particular)
\end{itemize}
\item every branch (unary, binary or ternary)  is labelled by a rule name. 
\item some branches (named $\imp_e,\lor_e,\exists_e$)  include an index which also appears on the hypotheses which are cancelled during the application of the rule. 
\end{itemize} 

If the multiset of free hypotheses of a proof $d$ is $\Gamma$ and the conclusion of $d$ is $C$, then $d$ is a proof of $\Gamma\seq C$ that is a proof of $C$ under the (conjunction of the) assumptions $\Gamma$.

If a rule says that  $H$ is cancellable in $d$ then any number of free occurrences of $H$ can be cancelled (one also says  discharged).
The cancelled hypotheses and the rule receive a fresh new index that encodes this fact (as this information is not recoverable from the proof tree,
this is why natural deductions are more than trees).

$$
\begin{array}{p{11.8ex}|rc|cl}
&\mbox{Introduction rules} & && \mbox{Elimination rules} \\[1ex] \hline &&&& \\  
Implication
&
\infer[\imp_i\alpha]{A\imp B}{\infer*{B}{A \Gamma [A]^{\alpha} \Delta [A]^{\alpha}}}
& & &
\infer[{\imp_e}]{B}
{A & A\imp B}\\[1ex] 
& 
\begin{minipage}{3cm}\it $A$ cancellable. 
\end{minipage} 
&&& \\[1ex] \hline &&&& \\  
Conjunction
&
\infer[\land_i]{A \land B}{A & B} 
&&&
\infer[\land_e^1]{A}{A\land B}
\quad 
\infer[\land_e^2]{B}{A\land B}
\\[1ex] \hline &&&& \\  
Falsum  & 
&&& 
\infer[\scriptstyle{\bot}]
{C}{\bot}
\\[1ex] 
&&&&\begin{minipage}{5cm}\it 
\emph{ex falso quod libet sequitur} 
\end{minipage} 
\\[1ex] \hline &&&& \\ 
\mbox{Disjunction} 
&
\infer[\lor_i^1]{A\lor B}{A} 
\quad 
\infer[\lor_i^2]{A\lor B}{B} 
&&&
\infer[{\lor_e \alpha}]{C}{
\infer*{A\lor B}{\Theta}
& \infer*[d_1]{C}{A \Gamma [A]^{\alpha} \Delta [A]^{\alpha}}
& \infer*[d_2]{C}{B \Theta [B]^{\alpha} \Phi [B]^{\alpha}}}
\\[1ex] 
&&&&\begin{minipage}{5cm}\it 
$A$  cancellable in $d_1$, $B$ in  $d_2$. 
\end{minipage} \\ \hline  
\end{array}
$$ 

$$
\begin{array}{p{2cm}|cc|cc}
&\mbox{Introduction rules} & && \mbox{Elimination rules} \\[1ex] \hline &&&& \\  
Existential\newline  quantifier
&
\infer
[\exists_i]
{\exists x A(x)} 
{A(t)}
&&&
\infer[{\exists_e \alpha}] 
{C}{\infer*{\exists x A(x)}{\Theta}  
&
\infer*{C}{A(u) \Gamma [A(x)]^{\alpha} \Delta [A(x)]^{\alpha}} 
}
\\[1ex] 
&&&&
\begin{minipage}{5cm}\it Afterwards, no free $x$ in $C$ nor in any free hypothesis. 
\end{minipage} 
\\ &&&& \\  \hline  &&&& \\ 
Universal\newline  quantifier
& 
\infer[{\forall_i}]{\forall x A(x)} 
{\infer*{A(x}{\Gamma}}
&&&
\infer[{\forall_e}]
{A(t)}{\forall x A(x)}
\\[1ex] 
& \begin{minipage}{2.5cm}\it No free $x$ in $\Gamma$. \end{minipage} &&&
\\ &&&& \\  \hline 
\end{array} 
$$ 

\normalsize

Natural deduction is ``naturally intuitionistic'': formul\ae\  like $\lnot\lnot A \imp A$ or $(\lnot X)\lor X$ are not provable. 
It is equivalent to other formulations of intuitionistic logic like
the sequent calculus  with many hypothesis and one conclusion.

There are no rules for negation $\lnot X$, which is treated as a shorthand for  $X\imp \bot$: 
$$\lnot X \equiv^{\textit{def}} X \imp \bot$$ For convenience, we
will sometimes write the negation rules as follows.

\[
\begin{array}{ccc}
\infer[\lnot_i\alpha]{\lnot A}{\infer*{\bot}{A\Gamma[A]^{\alpha}
\Delta[A]^{\alpha}}}
&&
\infer[\lnot_e]{\bot}{A & \lnot A}
\end{array}
\]

The reader can easily verify that given $\lnot A \equiv^{\textit{def}}
A \imp \bot$, these are just instances of $\imp_e$ and
$\imp_i$ where the subformula $B$ of the $\imp${} rules is $\bot$.

\section{Three formulations of classical logic} 

To obtain a natural deduction calculus for classical logic, one has to
add a family of proper axioms (i.e.\ axioms other than $A$, that is
$A\seq A$, which unfortunately complicates normalisation and the proof of the subformula property). 

$$
\begin{array}{p{4cm}cp{4cm}}
Tertium Non Datur & \infer{A \lor \lnot A}{tnd} & for all formula $A$\\ 
&\\ 
Reductio Ad Absurdum & \infer{\lnot \lnot A \imp A}{raa}& for all formula $A$\\ 
&\\ 
Pierce law & \infer{((P\imp Q)\imp P)\imp P}{Pierce} & for all formul\ae\  $P$ and $Q$
\end{array}
$$

RAA can also be expressed as rule cancelling several occurrences of $\lnot A$. 
\[
\infer[raa'\alpha]{A}
{\infer*{\bot}{\lnot A \Gamma [\lnot A]^{\alpha} \Delta [\lnot A]^{\alpha}}}
\]

$raa'$ is clearly equivalent to the axiom $raa$ given above: 
using the invertible rule  $\imp_i$ one obtains $\lnot\lnot A$ 
and then obtains $A$ by $raa$ above, as follows. 

\[
\infer[\imp_e]{A}{\infer[\imp_i\alpha]{\lnot\lnot A}
{\infer*{\bot}{\lnot A \Gamma [\lnot A]^{\alpha} \Delta [\lnot
    A]^{\alpha}}} & \infer[raa]{\lnot\lnot\ A \imp A}{}}
\]

Given $raa'$, $raa$ becomes derivable as follows.

\[
\infer[\imp_i(2)]{\lnot\lnot A \imp
  A}{\infer[raa'(1)]{A}{\infer{\bot}{[\lnot\lnot A]^2 [\lnot A]^1}}}
\]

\section{Equivalence of the three formulations of classical logic in intuitionistic logic} 

\subsection{Reductio ad Absurdum entails Tertium Non Datur}

\newcommand\argdeux{
\infer[{\imp_i \scriptstyle{1}}]
{\lnot \lnot a} 
{\infer[{\imp_e}]{\bot}{[\lnot(a\lor \lnot a)]^3 & \infer[{\lor_i^2}]{a\lor \lnot a}{[\lnot a]^1}}}
}

\newcommand\argun{
\infer[{\imp_i \scriptstyle{2}}]{\lnot a} 
{\infer[{\imp_e}]
{\bot} 
{[\lnot(a\lor \lnot a)]^3 
& \infer[{\lor_i^1}]{a\lor \lnot a} {[a]^2} 
}}}

$$
\infer[{\imp_e}] 
{a\lor \lnot a} 
{
\infer[{\imp_i \scriptstyle{3}}]{\lnot\lnot (a\lor \lnot a)}
{\infer[{\imp_e}] {\bot} 
{\argun & \argdeux}}
&\hspace*{-5ex}  \infer{\lnot\lnot (a\lor \lnot a)\imp (a\lor \lnot a)}{raa}
}
$$

\subsection{Tertium Non Datur entails Pierce law}

\newcommand\treeone{
\infer[\imp_i]{((p\imp q)\imp p) \imp p}{[p]^3}
}

\newcommand\treetwo{
\infer[{\imp_i \scriptstyle{1}}]
{p\imp q}
{\infer[\scriptstyle{\bot}]{q}{\infer[{\imp_e}]{\bot}{{[p]^1} & {[\lnot p]^3}}}}
}

\newcommand\treetwoplus{
\infer[{\imp_e \scriptstyle{2}}]{p}{ \treetwo & {[(p\imp q)\imp p]^2}  }
}

\newcommand\treetwofinal{
\infer[\imp_i \scriptstyle{2}]{((p\imp q)\imp p)\imp p}{\treetwoplus}
}

$$
\infer[\lor_e \scriptstyle{3}]{((p\imp q)\imp p)\imp p}
{\infer{p\lor \lnot p}{tnd} & \treeone & \treetwofinal} 
$$

\subsection{Pierce law entails Reductio ad Absurdum}

\newcommand\arbreun{
\infer[{\imp_i \scriptstyle{1}}]{\lnot P \imp P}{\infer[{\scriptstyle{\bot}}]{P}{\infer[\imp_e]{\bot}{[\lnot \lnot P]^2 & [\lnot P]^1}}}
}

$$
\infer[{\imp_i \scriptstyle{2}}]{\lnot\lnot P \imp P}
{\infer[{\imp_e}]{P}{\infer{(\lnot P \imp P)\imp P}{Pierce\ avec\ Q=\bot} & \arbreun}}
$$

\section{The G\"odel-Kolmogorov translation}

The not-not translation $F\nn$ of a formula $F$ is inductively defined as follows: 

\begin{center} 
\begin{tabular}{r@{\,=\,}l}
$\bot\nn$ & $\bot$ \\
$a\nn$ & $\lnot\lnot a$ \\[2mm]

$(A \land B)\nn$ & $A\nn \land B\nn$ \\ 
$(A \imp B)\nn$ & $A\nn \imp B\nn$ \\
$(\forall x. A)\nn$ & $\forall x. A\nn$ \\[2mm]

$(A \lor B)\nn$ & $\lnot\lnot(A\nn \lor B\nn)$ \\
$(\exists x. A)\nn$ & $\lnot\lnot \exists x. A\nn$ \\
\end{tabular}
\end{center} 

\LJ\  stands for plain natural deduction, which is intuitionistic. 

\LK\ stands for classical natural deduction, that is \LJ\ enriched by
one of the families of axioms given above (or all of them, since they are equivalent):  tertium non datur, reductio ad absurdum or Pierce law. 

Since $\bot\nn=\bot$ and $\lnot A=(A\imp \bot)$, the definition of the not not translation of an implicative formula yields the following remark: 

\begin{remark} 
$(\lnot A)\nn=\lnot(A\nn)$. 
\end{remark}

\begin{proposition}\label{nnnAyieldsnA} 
$\lnot\lnot\lnot A\seq^{\LJ} \lnot A$
\end{proposition}

\begin{proof} Here is the natural deduction proof of it: 
$$ 
\infer[\lnot_i(2)]{\lnot A}{
   \infer[\lnot_e]{\bot}{
      \lnot\lnot\lnot A 
    & \infer[\lnot_i(1)]{\lnot\lnot A}{
          \infer[\lnot_e]{\bot}{
             [A]^2 
           & [\lnot A]^1
          }
       }
   }
}
$$
\end{proof}

\begin{lemma}
 $\lnot\lnot F\nn\seq^{\LJ} F\nn$
\label{lemme} 
\end{lemma}

\begin{proof}
We proceed by induction on $F$. 

\begin{enumerate}
\item If $F=\bot$ one has to show that  $\lnot\lnot \bot \seq \bot$: 

$$\infer[\lnot_e]{\bot}{
   \lnot\lnot \bot
 & \infer[\lnot_i(1)]{\lnot \bot}{    [\bot]^1   }
}
$$

\item If $F=a$ we have to show $\lnot\lnot\lnot\lnot at \seq
  \lnot\lnot at$ which is a consequence of Proposition~\ref{nnnAyieldsnA} with $A=\lnot at$. 

\item  If $F=X\lor Y$ we have to show that $\lnot\lnot\lnot\lnot(X\nn
  \lor Y\nn) \seq \lnot\lnot(X\nn \lor Y\nn)$, which  is a consequence
  of Proposition~\ref{nnnAyieldsnA} with $A=\lnot(X\nn \lor Y\nn)$.

\item If $F=\exists x\ P$ one has to show that
  $\lnot\lnot\lnot\lnot(\exists x\ P\nn) \seq \lnot\lnot(\exists x\ P)$ which  is a consequence of Proposition~\ref{nnnAyieldsnA} with $A=\lnot(\exists x\ P\nn)$. 

\item If $F=A \imp B$, one has to show that $ \lnot\lnot (A\nn \imp B\nn)\seq (A\nn \imp B\nn)$.  The induction hypothesis (IH) makes sure that 
 $\lnot\lnot B\nn \seq B\nn$: 

$$\infer[\imp_i(3)]{A\nn \imp B\nn}{
   \infer*[IH]{B\nn}{
      \infer[\lnot_i(2)]{\lnot\lnot B\nn}{
           \infer[\lnot_e]{\bot}{
              \lnot\lnot (A\nn \imp B\nn)
            & \infer[\lnot_i(1)]{\lnot (A\nn \imp B\nn)}{
                 \infer[\lnot_e]{\bot}{
                    \infer[\imp_e]{B\nn}{
                       [A\nn]^3
                     & [A\nn \imp B\nn]^1
                    }
                  & [\lnot B\nn]^2
                 }
              }
          }
      }
   }
}
$$

\item If $F=A \land B$, one has to show that $ \lnot\lnot (A\nn \land B\nn)\seq (A\nn \land B\nn)$ and because of the induction hypothesis (IH) we can assume that 
$\lnot\lnot A\nn \seq A\nn$ and $\lnot\lnot B\nn \seq B\nn$.

$$   \infer*[IH]{A\nn}{
      \infer[\lnot_i(2)]{\lnot\lnot A\nn}{
          \infer[\lnot_e]{\bot}{
             \infer[\lnot_i(1)]{\lnot (A\nn \land B\nn)}{
                \infer[\lnot_e]{\bot}{
                   \infer[\land_e]{A\nn}{[A\nn \land B\nn]^1}
                 & [\lnot A\nn]^2
                }
             }
           & \lnot\lnot (A\nn \land B\nn)
          }
      }
   }
   $$
 
 \vspace{-5ex} 
 
$$\infer*[IH]{B\nn}{
      \infer[\lnot_i(4)]{\lnot\lnot B\nn}{
          \infer[\lnot_e]{\bot}{
             \infer[\lnot_i(3)]{\lnot (A\nn \land B\nn)}{
                \infer[\lnot_e]{\bot}{
                   \infer[\land_e]{B\nn}{[A\nn \land B\nn]^3}
                 & [\lnot B\nn]^4
                }
             }
           & \lnot\lnot (A\nn \land B\nn)
          }
      }
   }
   $$

\noindent From those two proofs, both with the single undischarged hypothesis\linebreak $\lnot\lnot (A\nn \land B\nn)$, one easily gets $A\nn \land B\nn$ by the rule $\land_i$. \medskip 

\item If $F=\forall x\ A$
one has to show that $\lnot\lnot (\forall x\ A\nn)\seq (\forall x\ A\nn)$ and the induction hypothesis (IH) guarantees  that
$\lnot\lnot A\nn \seq A\nn$. 

$$
\infer[\forall_i]{\forall x. A\nn}{
   \infer*[IH]{A\nn}{
      \infer[\lnot_i(2)]{\lnot\lnot A\nn}{
         \infer[\lnot_e]{\bot}{
            \lnot\lnot \forall x.A\nn
          & \infer[\lnot_i(1)]{\lnot \forall x.A\nn}{
               \infer[\lnot_e]{\bot}{
                  \infer[\forall_e]{A\nn}{[\forall x. A\nn]^1}
                & [\lnot A\nn]^2
               }
            }
          }  
      }
   }
}
$$

\noindent Notice that the hypothesis  $[\lnot A\nn]^2$ is cancelled before the $\forall_i$ rule is applied. 
\end{enumerate} 
\end{proof}

\begin{theorem}\label{goedel_kolmogorov} 
$\Gamma\nn\seq^{\LJ} F\nn$ if and only if $\Gamma\seq^{\LK} F$. 
\end{theorem}

\subsection{If $\Gamma\nn\seq^{\LJ} F\nn$ then $\Gamma\seq^{\LK} F$}

If $\Gamma\nn\seq^{\LJ} F\nn$ is provable in \LJ, then it is also provable in 
 \LK: indeed the rules of  \LJ{} are rules of  \LK. 

Since  in  \LK{} 
it possible to add and delete 
double negations, for every formula  $A$, both 
$A \vdash^{\LK} A\nn$ and  $A\nn  \vdash^{\LK} A$ hold, 
as an easy induction on the formula shows. 
Thus, a proof in \LK{} can be constructed. 

\begin{center} 
\begin{tabular}{c@{$\quad\leadsto\quad$}c@{$\quad\leadsto\quad$}c}
\infer*[\LJ]{F\nn}{A_1\nn & \ldots & A_n\nn} &
\infer*[\LK]{F\nn}{A_1\nn & \ldots & A_n\nn} &
\infer*[\LK]{F}{\infer*[\LK]{F\nn}{\infer*[\LK]{A_1\nn}{A_1}\!\!\!\!\!\!\! & \ldots & \infer*[\LK]{A_n\nn}{A_n}}} \\
\end{tabular}
\end{center} 

\subsection{If $\Gamma\seq^{\LK} F$ then $ \Gamma\nn\seq^{\LJ} F\nn$}

We proceed by induction on the height of the proof in  \LK. 
Observe that the obtained 
$\LJ$ proof has the same occurences of free variables.

\subsubsection{The height of the proof is  $0$ and the proof is an axiom  $A\seq^{\LK} A$}

\begin{center}
\begin{tabular}{c@{$\quad\leadsto\quad$}c}
$A$ & $A\nn$ \\
\end{tabular}
\end{center}

\subsubsection{The height of the proof is $0$  and it is an application tertium non datur  $\seq^{\LK} A\lor \lnot A$}

Remember that $(\lnot A)\nn=\lnot(A\nn)$ . 

$$
\infer[\lnot_i(2)]{\lnot\lnot (A\nn \lor \lnot A\nn)}{
   \infer[\lnot_e]{\bot}{
      \infer[\lor_i]{A\nn \lor \lnot A\nn}{
         \infer[\lnot_i(1)]{\lnot A\nn}{
            \infer[\lnot_e]{\bot}{
               \infer[\lor_i]{A\nn \lor \lnot A\nn}{
                  [A\nn]^1
               }
            & [\lnot (A\nn \lor \lnot A\nn)]^2
            }
         }
      }
    & [\lnot (A\nn \lor \lnot A\nn)]^2
   }
}
$$

\subsubsection{The hight of the proof is  $0$ and it comes from  reductio ad absurdum $\seq^{\LK} (\lnot \lnot A)\imp A$}

We have to show that $\seq (\lnot\lnot A)\nn \imp A\nn$ but since $\lnot A=(A\imp \bot)$ 
we know that  $(\lnot\lnot A)\nn=\lnot\lnot A\nn$. We therefore have
to show that  $\seq \lnot\lnot (A\nn) \imp A\nn$, but this true by Lemma~\ref{lemme}.

\subsubsection{The proof ends with $\bot_e$}

We apply the induction hypothesis  (IH) to the proof without this last rule, 
using the fact that  $\bot\nn=\bot$. 

\begin{center}
\begin{tabular}{c@{$\quad\leadsto\quad$}c}
\mbox{
 \infer[\bot_e]{A}{\infer*{\bot}{\Gamma}}
}
&
\mbox{
 \infer[\bot_e]{A\nn}{\infer*[IH]{\bot}{\Gamma\nn}}
}
\end{tabular}
\end{center}

\subsubsection{The proof ends with the rule  $\imp_e$}

The induction hypothesis (IH) can be applied to the two proofs 
obtained by suppressing this last rule. 

\begin{center}
\begin{tabular}{c@{$\quad\leadsto\quad$}c}
\mbox{
\infer[\imp_e]{B}{\infer*{A}{\Gamma} & \infer*{A\imp B}{\Delta}}
} &
\mbox{
\infer[\imp_e]{B\nn}{\infer*[IH]{A\nn}{\Gamma\nn} & \infer*[IH]{A\nn\imp B\nn}{\Delta\nn}}
} 
\end{tabular}
\end{center}

\subsubsection{The last rule is  $\imp_i$}

The induction hypothesis (IH) can be applied to the proof 
obtained by suppressing this last rule. 

\begin{center}
\begin{tabular}{c@{$\quad\leadsto\quad$}c}
\mbox{
\infer[\imp_i(k)]{A\imp B}{\infer*{B}{\Gamma & [A]^k}}
} &
\mbox{
\infer[\imp_i(k)]{A\nn\imp B\nn}{\infer*[IH]{B\nn}{\Gamma\nn & [A\nn]^k}}
}
\end{tabular}
\end{center}

\subsubsection{The last rule is  $\land_e$}

The induction hypothesis (IH) can be applied to the proof 
obtained by suppressing this last rule. 

\begin{center}
\begin{tabular}{c@{$\quad\leadsto\quad$}c}
\mbox{
\infer[\land_e]{A}{\infer*{A\land B}{\Gamma}}
} &
\mbox{
\infer[\land_e]{A\nn}{\infer*[IH]{A\nn \land B\nn}{\Gamma\nn}}
} 
\end{tabular}
\end{center}

\subsubsection{The last rule is $\land_i$}

The induction hypothesis (IH) can be applied to the two proofs 
obtained by suppressing this last rule.

\begin{center}
\begin{tabular}{c@{$\quad\leadsto\quad$}c}
\mbox{
\infer[\land_i]{A\land B}{\infer*{A}{\Gamma} & \infer*{B}{\Delta}}
}
&
\mbox{
\infer[\land_i]{A\nn\land B\nn}{\infer*[IH]{A\nn}{\Gamma\nn} & \infer*[IH]{B\nn}{\Delta\nn}}
}
\end{tabular}
\end{center}

\subsubsection{The last rule is  $\lor_e$}

The induction hypothesis (IH) can be applied to the three proofs 
obtained by suppressing this last rule. 
We use Lemma~\ref{lemme}.

\begin{center}
\begin{tabular}{c@{$\quad\leadsto\quad$}c}
\mbox{
\infer[\lor_e]{C}{\infer*{A\lor B}{\Gamma} & \infer*{C}{\Delta & [A]^k} & \infer*{C}{\Theta & [B]^k}}
}
& \\[6mm]
\multicolumn{2}{c}{
\mbox{
\infer*[\textit{Lemma}~\ref{lemme}]{C\nn}{
\infer[\lnot_i(m)]{\lnot\lnot C\nn}{
\infer[\lnot_e]{\bot}{
\infer[\lnot_i(l)]{\lnot(A\nn \lor B\nn)}{
\infer[\lnot_e]{\bot}{
    \infer[\lor_e]{C\nn}{[A\nn\lor B\nn]^l & \infer*[IH]{C\nn}{\Delta\nn & [A\nn]^k} & \infer*[IH]{C\nn}{\Theta\nn & [B\nn]^k}}
 &  [\lnot C\nn]^m
}} &
\infer*[IH]{\lnot\lnot(A\nn\lor B\nn)}{\Gamma\nn}
}}}
}}
\end{tabular}
\end{center}

\subsubsection{The last rule is $\lor_i$}

The induction hypothesis (IH) can be applied to the proof 
obtained by suppressing this last rule.

\begin{center}
\begin{tabular}{c@{$\quad\leadsto\quad$}c}
\mbox{
\infer[\lor_i]{A\lor B}{\infer*{A}{\Gamma}}
}
&
\mbox{
\infer[\lnot_i(l)]{\lnot\lnot (A\nn\lor B\nn)}{
  \infer[\lnot_e]{\bot}{
   \infer[\lor_i]{A\nn\lor B\nn}{\infer*[IH]{A\nn}{\Gamma\nn}}
 & [\lnot (A\nn\lor B\nn)]^l
}}}
\end{tabular}
\end{center}

\subsubsection{The last rule is $\forall_e$}

The induction hypothesis (IH) can be applied to the proof 
obtained by suppressing this last rule. 

\begin{center}
\begin{tabular}{c@{$\quad\leadsto\quad$}c}
\mbox{
\infer[\forall_e]{A[x:=t]}{\infer*{\forall x. A}{\Gamma}}
} &
\mbox{
\infer[\forall_e]{A[x:=t]\nn}{\infer*[IH]{\forall x. A\nn}{\Gamma\nn}}
} 
\end{tabular}
\end{center}

\subsubsection{The last rule is  $\forall_i$}

The induction hypothesis (IH) can be applied to the proof 
obtained by suppressing this last rule. 

\begin{center}
\begin{tabular}{c@{$\quad\leadsto\quad$}c}
\mbox{
\infer[\forall_i]{\forall x.A}{\infer*{A}{\Gamma}}
}
&
\mbox{
\infer[\forall_i]{\forall x.A\nn}{\infer*[IH]{A\nn}{\Gamma\nn}}
}
\end{tabular}
\end{center}

\subsubsection{The last rule is  $\exists_e$}

The induction hypothesis (IH) can be applied to the two proofs 
obtained by suppressing this last rule. 
We use Lemma~\ref{lemme}.

\begin{center}
\begin{tabular}{c@{$\quad\leadsto\quad$}c}
\mbox{
\infer[\exists_e(k)]{C}{\infer*{\exists x. A}{\Gamma} & \infer*{C}{\Delta &[A]^k}}
} &
\mbox{
\infer*[\textit{Lemma}~\ref{lemme}]{C\nn}{
\infer[\lnot_i(m)]{\lnot\lnot C\nn}{
   \infer[\lnot_e]{\bot}{
     \infer[\lnot_i(l)]{\lnot \exists x. A\nn}{ 
        \infer[\lnot_e]{\bot}{
           \infer[\exists_e(k)]{C\nn}{
               [\exists x. A\nn]^l 
             & \infer*[IH]{C\nn}{
                  \Delta\nn 
                & [A\nn]^k
               }
           }
      & [\lnot C]^m
        }
     }
   & \infer*[IH]{\lnot\lnot \exists x. A\nn}{\Gamma\nn}
  }
}
}}
\end{tabular}
\end{center}

\subsubsection{The last rule is $\exists_i$}

The induction hypothesis (IH) can be applied to the proof 
obtained by suppressing this last rule.

\begin{center}
\begin{tabular}{c@{$\quad\leadsto\quad$}c}
\mbox{
\infer[\exists_i]{\exists x.A}{\infer*{A[x:=t]}{\Gamma}}
} &
\mbox{
\infer[\lnot_i(l)]{\lnot\lnot\exists x.A\nn}{
   \infer[\lnot_e]{\bot}{
       \infer[\exists_i]{\exists
         x.A\nn}{\infer*[IH]{A[x:=t]\nn}{\Gamma\nn}}
     & [\lnot \exists x. A\nn]
   }}} \\
\end{tabular}
\end{center}

\backmatter 

\bibliographystyle{smfplain}
\bibliography{bigbiblio} 

\providecommand{\bysame}{\leavevmode ---\ }
\providecommand{\og}{``}
\providecommand{\fg}{''}
\providecommand{\smfandname}{\&}
\providecommand{\smfedsname}{\'eds.}
\providecommand{\smfedname}{\'ed.}
\providecommand{\smfmastersthesisname}{M\'emoire}
\providecommand{\smfphdthesisname}{Th\`ese}
\begin{thebibliography}{10}

\bibitem{vanDalen2013}
{\scshape D.~van Dalen} -- \emph{Logic and structure}, fifth \smfedname,
  Universitext, Springer-Verlag, 2013.

\bibitem{Gen34}
{\scshape G.~Gentzen} -- {\og Untersuchungen {\"u}ber das logische {S}chlie\ss
  en {I}\fg}, \emph{Mathematische {Z}eitschrift} \textbf{39} (1934),
  p.~176--210, Traduction Fran\c caise de R.~Feys et J.~Ladri\`ere: Recherches
  sur la d\'eduction logique, Presses Universitaires de France, Paris, 1955.

\bibitem{Gen34b}
\bysame , {\og Untersuchungen {\"u}ber das logische {S}chlie\ss en {II}\fg},
  \emph{Mathematische Zeitschrift} \textbf{39} (1934), p.~405--431, Traduction
  fran\c caise de J.~Ladri\`ere et R.~Feys: Recherches sur la d\'eduction
  logique, Presses Universitaires de France, Paris, 1955.

\bibitem{GLT88}
{\scshape J.-Y. Girard, Y.~Lafont {\normalfont \smfandname} P.~Taylor} --
  \emph{Proofs and types}, Cambridge Tracts in Theoretical Computer Science,
  no.~7, Cambridge University Press, 1988.

\bibitem{Glivenko29}
{\scshape V.~Glivenko} -- {\og Sur quelques points de la logique de {M}.
  {B}rouwer\fg}, \emph{Bulletin de la Societ\'e Mathematique de Belgique}
  \textbf{15} (1929), p.~183--188.

\bibitem{Goedel1933HA}
{\scshape K.~G{\"o}del} -- {\og Eine interpretation des intuitionistischen
  aussagenkalk\"uls\fg}, \emph{Erg. Math. Kolloqu.} \textbf{4} (1933),
  p.~39--40 (German).

\bibitem{Goedel33nonnon}
\bysame , {\og Zur intuitionistischen arithmetik und zahlentheorie\fg},
  \emph{Ergebnisse eines mathematischen Kolloquiums} \textbf{4} (1933),
  p.~34--38 (German).

\bibitem{Kolmogorov1925}
{\scshape A.~Kolmogorov} -- {\og Sur le principe de tertium non datur\fg},
  \emph{Mathematicheskii Sbornik} \textbf{32} (1925), no.~4, p.~646--667.

\bibitem{Mints2000}
{\scshape G.~Mints} -- \emph{A short introduction to intuitionistic logic},
  University Series in Mathematics, Sringer, 2000.

\bibitem{Pra65}
{\scshape D.~Prawitz} -- \emph{{Natural Deduction, a Proof-theoretical Study}},
  Acta universitatis stockholmiensis --- Stockholm studies in philosophy,
  no.~3, Almqvist and Wiksell, Stockholm, 1965.

\bibitem{RasiowaSikorski1963}
{\scshape H.~Rasiowa {\normalfont \smfandname} R.~Sikorski} -- \emph{The
  mathematics of metamathematics}, Monografie matematyczne, vol.~41, Polish
  Scientific Publishers, 1963.

\bibitem{TD88}
{\scshape A.~Troelstra {\normalfont \smfandname} D.~van Dalen} --
  \emph{Constructivism in mathematics (vol. 1)}, Studies in Logic and the
  founadations of mathematics, vol. 121, North-Holland, 1988.

\bibitem{Frege2Goedel1967}
{\scshape J.~{van Heijenoort}} (\smfedname) -- \emph{From frege to g\"odel. a
  source book in mathematical logic, 1879--1931.}, Cambridge, MA: Harvard
  University Press, 1967 (English).

\end{thebibliography}
\end{document}